\documentclass[a4paper, 12pt]{article}

\usepackage[dvips]{graphics,color}
\def\spc{\color[rgb]{1,0.2,0.2}}
\def\nc{\normalcolor}
\usepackage[british]{babel}
\usepackage[pdftex]{graphicx}
\usepackage{graphics, setspace}

\usepackage{amsmath, amsthm, amsfonts, amsbsy, amssymb, scrextend, graphicx}
\usepackage{hyperref} 
\usepackage{verbatim}
\usepackage{color}
\usepackage{epstopdf}

\usepackage{color}

\usepackage{tikz}
\usepackage[export]{adjustbox}
\usepackage{float}

\newtheorem {theorem}{Theorem}

\newcommand{\R}{{\mathbb{R}}}

\newcommand{\eps}{\varepsilon}

\begin{document}

\date{}

\title{Skew Brownian motion with dry friction: joint density approach}

\author{
Alexander Gairat\footnote{
Gearquant. Email address: agairat@gearquant.com
}\, and 
Vadim Shcherbakov\footnote{
 Royal Holloway,  University of London.
 Email address: vadim.shcherbakov@rhul.ac.uk
}
}

\maketitle

\begin{abstract}
{\small 
This note concerns the distribution of Skew Brownian motion with dry friction and its occupation time. 
 These distributions were  obtained in~\cite{Berezin} by using the  Laplace transform and joint characteristic functions. 
We provide an alternative approach, which is  based on the use of the joint density 
 for Skew Brownian motion, its last visit to the origin, its
local and occupation times derived in~\cite{Gairat}.
}
\end{abstract}

\noindent {{\bf Keywords:} Skew Brownian motion, Caughey-Dienes process, local time,  occupation time }

\section{Introduction}

Let $X_t=(X_t,\, t\geq 0)$  be a continuous time stochastic process defined as a 
solution of the following stochastic differential equation 
\begin{equation}
\label{X0}
X_t=X_0+ \int\limits_0^tm(X_s)ds+(2p-1)L_t+W_t,
\end{equation}
where $p\in (0, 1)$,
\begin{equation}
\label{m}
m(x)=m_11_{\{x\geq 0\}}+m_21_{\{x<0\}},\, x\in \R,
\end{equation}
for some constants $m_1$ and $m_2$,
\begin{equation}
\label{L}
L_t=\lim\limits_{\eps\to 0}\frac{1}{2\eps}\int\limits_0^t 1_{\{-\eps\leq X_s
\leq \eps\}}ds
\end{equation}
is the  local time of the process at zero, and $(W_t,\, t\geq 0)$ is standard Brownian motion (BM).
The process $X_t$ is well known. If $p=\frac{1}{2}$ and $m_1=m_2=0$, then 
$X_t$ is  standard BM.
If $m_1=m_2=0$,  then it is  Skew Brownian motion (SBM)  with parameter $p$
(e.g. see the survey~\cite{Lejay} and references therein).  
The distribution of the process $X_t$ and its functionals (e.g. local and occupation times etc) 
is of great interest in applications and attracted the attention of many researchers. 
For example, 
the trivariate density of BM, its local and occupation time  was obtained 
in~\cite{Karatzas} and applied to problems of stochastic control.
 The distribution of the process $X_t$, its local and occupation time
was obtained  in~\cite{Appu}  in the case $m_1=m_2$ and used to explain results of some
laboratory experiments for an advection-dispersion phenomenon. 
The process $X_t$ with the piecewise linear drift~\eqref{m} naturally appeared
in~\cite{Gairat} in the study of a two-valued local volatility model.
The latter is a generalisation of the log-normal model for the underlying price 
 on the case when the volatility of the price can take two different values. 
The joint density  of the process, its local and occupation times, and the last visit to the origin 
was obtained in~\cite{Gairat}  in an exact analytical form.  This result 
was  applied in that  paper to generalize  the Black-Scholes formula for the option price
on the case of the two-valued volatility.

In the case  \(m_2=-m_1=m\) the process  $X_t$ is also known as the   skew Caughey-Dienes process, 
or SBM with dry friction (e.g., see~\cite{Berezin} and references therein).  
Densities of both the skew Caughey-Dienes process
and its occupation time on the non-negative half-line were derived 
in~\cite{Berezin} by using the Laplace transform and joint characteristic functions,
which requires rather heavy  computations.
In this note we show  how  these distributions can be alternatively obtained  by using the results of~\cite{Gairat}.

\section{Results}

\subsection{The distribution of the skew Caughey-Dienes process}

In this section we derive the distribution of the skew Caughey-Dienes process.
This process is the solution of equation~\eqref{X0} in
the special case \(m_2=-m_1=m\), that is 
\begin{equation}
\label{X1}
dX_t=-m\cdot {\rm sgn}\left(X_t\right)+(2 p -1)dL_t+dW_t.
\end{equation}

\begin{theorem}
\label{T1}
Let $X_t=(X_t,\, t\geq 0)$  be the  solution of 
 equation~\eqref{X1}. 
If $X_0=0$, then given $T>0$ the density  of $X_T$ is 
\begin{equation}
\label{pdf-X}
\phi_T(x)
=\begin{cases}
2p\left(\frac{e^{-\frac{(m T+x)^2}{2 T}}}{\sqrt{2\pi T}}+\frac{m}{2} e^{-2 m x} 
\left[1+{\rm Erf}\left(\frac{m T-x}{\sqrt{2T}}\right)\right]\right),& \text{if } x\geq 0,\\
2q\left(\frac{e^{-\frac{(-m T+x)^2}{2 T}}}{\sqrt{2\pi T}} +\frac{m}{2} e^{2 m x} 
\left[1+{\rm Erf}\left(\frac{m T+x}{\sqrt{2T}}\right)\right]\right),& \text{if } x<0,
\end{cases}
\end{equation}
where $q=1-p$ and ${\rm Erf}(z)$ is the standard error function.
\end{theorem}
\begin{proof}
Fix $T>0$ and define the following quantities
\begin{align}
\label{tau}
\tau&=\max\{t\in (0,T]: X_t=0\},\\
\label{V}
V&=\int\limits_{\tau_0}^{\tau}1_{\{X_t\geq 0\}}dt,
\end{align}
where $\tau_0=\min \left\{t:X_t=0\right\}$,
and  the function 
\begin{equation}
\label{psi-p}
\psi_{p, T}(t,v, x, l)=\begin{cases}
2p\cdot h(v, lp)h(t-v, lq)h(T-t, x),& \text{if } x\geq 0,\\ 
2q\cdot h(v, lp)h(t-v, lq)h(T-t, x),& \text{if } x<0, 
\end{cases}
\end{equation} 
for $ 0\leq v\leq  t\leq T, l\geq 0$, 
where 
\begin{equation}
\label{h}
h(s, y)=\frac{|y|}{\sqrt{2\pi s^3}}e^{-\frac{y^2}{2s}},\,\, y\in \R, s\in \R_{+},
\end{equation}
is the  density  of the first passage time  to zero of the standard BM starting at $y$.
It was shown in~\cite[Theorem 2]{Gairat} that, if $X_0=0$, then 
 the joint  density of $\left(\tau, V, X_T, L_T\right)$ is 
 given by 
\begin{equation}
\label{phi}
\phi_{T}(t, v, x, l)
=\psi_{p, T}(t, v, x, l)e^{-\frac{m_1^2v+m^2_2(T-v)}{2}-l\left(m_1p-qm_2\right)+m(x)x},
\end{equation}
where $\psi_{p, T}(t,v,z,l)$ is defined by equation (\ref{psi-p}).
If $m_1=-m_2=-m$, then the function~\eqref{phi} simplifies as follows
\begin{equation}
\begin{split}
\phi_{T}(t, v, x, l)
&=\psi_{p, T}(t, v, x, l)e^{-\frac{m^2T}{2}-m\cdot x\cdot {\rm sgn}(x)+l\cdot m}\\
&=\begin{cases}
2p\cdot h(v, lp)h(t-v, lq)h(T-t, x)e^{-\frac{m^2T}{2}-m\cdot x+l\cdot m},&
 \text{if } x\geq 0,\\ 
2q\cdot h(v, lp)h(t-v, lq)h(T-t, x)e^{-\frac{m^2T}{2}+m\cdot x+l\cdot m},& \text{if } x<0,
\end{cases}
\end{split}
\end{equation}
for $ 0\leq v\leq  t\leq T, l\geq 0$.
Using the convolution property  of the hitting times of Brownian motion and the fact that $p+q=1$, 
 we have that 
$$\int_0^th(v, lp)h(t-v, lq)dv=h(t,l),$$
which gives 
the joint density of  $\left(\tau, X_T, L_T\right)$, namely, 
\begin{equation}
\phi_T(t,x,l)=\int_0^{t} \phi_T(t,v,x,l)dv=\begin{cases}
2p\cdot h(t,l )h(T-t,x)e^{-\frac{m^2T}{2}-x\cdot m+l\cdot m},& \text{if } x\geq 0,\\
2q\cdot h(t,l )h(T-t,x)e^{-\frac{m^2T}{2}+x\cdot m+l\cdot m},& \text{if } x\geq 0.
\end{cases}
\end{equation}
Using  the convolution property of the hitting times again gives that 
$$\int_0^Th(t,l )h(T-t,x)dt=h(T, l+|x|).$$ 
Therefore, the joint density  of $ X_T$ and $L_T^{\left(0\right)}$ is as follows
\begin{equation}
\phi_T(x,l)=\begin{cases}
2p\cdot  h(T,l+x )e^{-\frac{m^2T}{2}-m\cdot x+l\cdot m},& \text{ if  } x\geq 0,\\
2q\cdot  h(T,l-x )e^{-\frac{m^2T}{2}+m\cdot x+l\cdot m},& \text{ if  } x<0.
\end{cases}
\end{equation}
It is left to integrate out the local time in order to obtain the density  
of $X_T$. 
Integration gives that 
$$\int_0^{\infty}\phi_T(x,l)dl=\phi_T(x),$$
where $\phi_T(x)$ is the function defined in~\eqref{pdf-X}, as claimed.
\end{proof}

\subsection{The distribution of the occupation time}
Let $X_t=(X_t,\, t\geq 0)$ be the solution of equation~\eqref{X1}.
Given $T>0$ define 
\begin{equation}
\label{U}
U=\int\limits_{0}^T 1_{\{X_t\geq 0\}}dt,
\end{equation}
i.e. $U$ is  the occupation time  of the non-negative half-line during the time period 
$[0, T]$ (the occupation time).
In~\cite{Berezin} the density of the occupation time 
 is expressed in term of a double integral of a rather complicated function. 
We show that this density 
can be obtained  as  an  integral  of a function of one variable, which is explicitly 
 expressed in terms 
of the complementary error function ${\rm Erfc}(z)=1-{\rm Erf}(z)$.

Note first that  if $X_0=0$, then 
$U=V+T-\tau$, if $X_T\geq 0$, and $U=V$, if $X_T<0$, 
where quantities $\tau$ and $V$ are defined in~\eqref{tau} and~\eqref{V} respectively.
Therefore,  the joint density $\varphi_T(t,u, x, l)$ of  $(\tau, U, X_T, L_{T})$
is (see equation (15) in~\cite{Gairat}) 
\begin{equation}
\label{phi1}
\varphi_T(t,u, x, l)
=
\begin{cases}
2p\cdot h(u+t-T, lp)h(T-u, lq)h(T-t, x)e^{-\frac{m^2T}{2}-x\cdot m+l\cdot m},& \\
\mbox{if} \, \, x\geq 0,\,\, l>0,\,\, \mbox{and}\,\,t\leq T,\, T-t\leq u\leq T;&\\
2q\cdot h(u, lp)h(t-u, lq)h(T-t, x)e^{-\frac{m^2T}{2}+x\cdot m+l\cdot m},& \\
\mbox{if}\,  \, x< 0,\,\, l>0, \,\, \mbox{and} \,\, 0\leq u\leq t\leq T.
\end{cases}
\end{equation}
Using that 
\begin{align*}
\int_{T-u}^Th(u-T+t, lp)h(T-t, x)dt&=h(u, lp+x)\quad\text{for}\quad x\geq 0,\\
\int_{u}^Th(t-u, lq)h(T-t, x)dt&=h(T-u, lq+|x|)\quad\text{for}\quad x<0,
\end{align*}
we obtain the joint density $\varphi_T(u, x, l)$ of  $(U, X_T, L_{T})$.
Integrating over the variable $x$ gives the joint density of the occupation time $U$ and
 the local time $L_{T}$, that is 
\begin{equation}
\label{pdf2}
\varphi_T(u, l)=2e^{-\frac{m^2T}{2}+lm}\left(pF(u, l, p)+qF(T-u, l, q)\right),
\end{equation}
where 
\begin{equation}
F(y, l, c)=\frac{e^{-\frac{l^2c^2}{2y}}}{\sqrt{2\pi y}}-\frac{m}{2}
{\rm Erfc}\left[\frac{lc+my}{\sqrt{2y}}\right]e^{lmc+\frac{m^2}{2}y}.
\end{equation}
Thus, the density  of the occupation time is given by 
the  integral $\displaystyle{\varphi_T(u)=\int_0^{\infty}\varphi_T(u, l) dl}$, 
where  the function $\varphi_T(u, l)$  is explicitly expressed 
in terms of the complementary error function, as claimed.

\end{document}